\documentclass[11pt, a4paper]{article}
\usepackage[utf8]{inputenc}
\usepackage{mathtools}
\usepackage{amsmath}
\usepackage{amsthm}
\usepackage{algorithm}
\usepackage{algorithmic}
\usepackage{amssymb}
\usepackage{parskip}
\usepackage{stackengine}
\mathtoolsset{showonlyrefs}
\usepackage[dvipsnames]{xcolor}
\usepackage{tikz}
\usepackage{pgfplots}
\pgfplotsset{major grid style={dashed}}
\pgfplotsset{compat=1.18}
\usepackage[a4paper, margin=3cm]{geometry}
\usepackage{comment}
\usepackage{mathrsfs} 
\usepackage{dsfont}
\usepackage{esint}
\usepackage[export]{adjustbox}
\usepackage{multirow}
\usepackage{mathtools}
\usepackage{enumitem}
\usepackage{scalerel}
\usepackage{color}
\definecolor{hanblue}{rgb}{0.27, 0.42, 0.81}
\definecolor{mordantred19}{rgb}{0.68, 0.05, 0.0}
\definecolor{red}{rgb}{0.68, 0.05, 0.0}
\definecolor{green}{rgb}{0.0, 0.5, 0.0}
\usepackage[colorlinks, citecolor=blue,linkcolor=red]{hyperref}
\usepackage{subcaption}

\newcommand{\R}{\mathbb{R}}

\newcommand\restr[2]{{
  \left.\kern-\nulldelimiterspace
  #1 
  \vphantom{\big|} 
  \right|_{#2}
  }}

\allowdisplaybreaks

\theoremstyle{plain}
\newtheorem{thm}{Theorem}
\numberwithin{thm}{section}

\newtheorem{prop}[thm]{Proposition}

\theoremstyle{definition}
\newtheorem{defi}[thm]{Definition}

\newtheorem{rem}[thm]{Remark}

\renewcommand{\epsilon}{\varepsilon}

\renewcommand{\phi}{\varphi}
\renewcommand{\leq}{\leqslant}
\renewcommand{\geq}{\geqslant}

\usepackage{xcolor}
\definecolor{hpink}{RGB}{255, 20, 147}

\title{Perturbation-Aware Distributionally Robust Optimization for Inverse Problems}
\author{Floor van Maarschalkerwaart\thanks{Department of Applied Mathematics, University of Twente, Enschede, The Netherlands \\
(\texttt{f.vanmaarschalkerwaart@utwente.nl},  \texttt{c.brune@utwente.nl}, \texttt{m.c.carioni@utwente.nl})} 
\ \ Subhadip Mukherjee\thanks{Indian Institute of Technology (IIT), Kharagpur, India (\texttt{smukherjee@ece.iitkgp.ac.in})}
\ \ Malena Sabaté Landman\thanks{Mathematical Institute, University of Oxford, UK (\texttt{malena.sabatelandman@maths.ox.ac.uk})}\\ 
Christoph Brune\footnotemark[1], 
\ \ 
Marcello Carioni\footnotemark[1]}

\begin{document}

\maketitle

\begin{abstract}
This paper builds on classical distributionally robust optimization techniques to construct a comprehensive framework that can be used for solving inverse problems. Given an estimated distribution of inputs in $X$ and outputs in $Y$, an ambiguity set is constructed by collecting all the perturbations that belong to a prescribed set $K$ and are inside an entropy-regularized Wasserstein ball. 
By finding the worst-case reconstruction within $K$ one can produce reconstructions that are robust with respect to various types of perturbations: $X$-robustness, $Y|X$-robustness and, more general, targeted robustness depending on noise type, imperfect forward operators and noise anisotropies. 

After defining the general robust optimization problem, we derive its (weak) dual formulation and we use it to design an efficient algorithm. Finally, we demonstrate the effectiveness of our general framework to solve matrix inversion and deconvolution problems defining $K$ as the set of multivariate Gaussian perturbations in $Y|X$.  

\medskip

\textbf{Keywords.} Distributionally Robust Optimization, Inverse Problems, Duality.
\end{abstract}

\section{Introduction}

Optimization under uncertainty has always played a central role in decision-making processes across various fields such as operation research, engineering, economics, and machine learning. 
Classical methods assume that the probability distribution governing uncertain parameters is known. However, due to limited samples and imperfect measurements, it is very common in practice to require the uncertainties to be considered for the final decision. \textit{Distributionally robust optimization} (DRO) addresses this challenge by seeking solutions that are robust against distributional ambiguity. Instead of optimizing over a single distribution, DRO considers a family (or ambiguity set) of distributions and identifies decisions that perform well in the worst-case scenario within this set.

One prominent formulation of DRO uses the Wasserstein distance as a measure of proximity between probability distributions to define the ambiguity set. Indeed, the Wasserstein distance, rooted in optimal transport theory, is particularly attractive for DRO, due to its intuitive geometric interpretation and its ability to capture discrepancies in both support and mass between distributions. Moreover, it allows for smoothing strategies (e.g., entropic regularization) that make the computation of the Wasserstein distance more efficient and scalable. 

DRO frameworks are explicitly designed for optimizing in the presence of uncertainties. For this reason, it is natural to draw a connection to inverse problems, where the measurement process $H: X \rightarrow Y$ is often imperfect and the observation $y = Hx$ is corrupted by different types of noise. Despite this analogy, the relation between DRO approaches, inverse problems reconstruction, and regularization is currently substantially unexplored. In this paper, we aim to make a first step to bridge this gap by designing a general DRO framework that is flexible to be used in inverse problems. In particular, given an in-sample estimation of the data acquisition process $\mu^* \in P(X\times Y)$ (often given as an empirical estimation) we define a \emph{perturbation-aware} DRO framework, where the ambiguity set is constrained to all possible perturbations of $\mu^*$ that belong to a given class $K \subset P(X\times Y)$. This framework builds on Bayesian variants of DRO \cite{levy2012robust,zhang2022distributionally}. It allows - by making suitable choices of $K$ - to enforce the desired robustness in the reconstruction and is ultimately linked to specific choices for Tikhonov-type regularization. For example, by choosing $K \subset P(X)$, one could enforce robustness to perturbations of the input, while by choosing $K \subset \{\mu \in P(X \times Y): \mu_X^* \in P(X) \text{ is the } X\text{-marginal of } \mu\}$, one instead enforces robustness with respect to perturbations in the conditional distribution of $Y|X$, thus mimicking the uncertainty of the data acquisition process. Moreover, by further specifying $K$ one can recover robustness to prescribed distributions (e.g. Gaussian, Poisson) and impose desired anisotropies.

In practice, we define the \emph{perturbation-aware} DRO as a min-max problem where the ambiguity set includes all perturbations of $\mu^*$ in $K$ that additionally belong to the entropy-regularized Wasserstein ball. To implement an efficient numerical method to compute reconstructions of inverse problems we derive a (weak) dual formulation for the \emph{perturbation-aware} DRO. Although this provides only an upper bound for the problem, we believe that under suitable assumptions a strong duality result should hold. However, we leave this proof up to future research. This allows us to rely on a biased stochastic mirror descent approach \cite{nemirovskij1983problem}, where we follow \cite{wang2021sinkhorn} which was designed for entropy-regularized Wasserstein DRO with no perturbation constraints. 
To showcase the effectiveness of our framework for inverse problems, we present several examples. First, we set $K$ as the family of multivariate Gaussian perturbations in $Y|X$ with varying (scalar) variance. Subsequently, we consider the case where the Gaussian perturbations can have an arbitrary covariance matrix, allowing for anisotropic perturbations and robustness. Finally, we apply our method to solve image deconvolution problems with Gaussian and Poisson noise. \\

\textbf{Main contributions.} The main contributions of this paper can be summarized as follows:

\begin{enumerate}
    \item We introduce a \emph{perturbation-aware} DRO framework, where perturbations of an in-sample estimated distribution are constrained to be in a set $K \subset P(X\times Y)$.
    \item We show how this framework naturally applies to inverse problems and is amenable to different kinds of robustness, resulting in a flexible inverse that can be applied to measurements corrupted by any type of noise.
    \item We derive a weak duality formula for the problem, enabling the design of efficient numerical schemes. Under suitable assumptions, we believe strong duality will hold.
    \item We demonstrate the effectiveness of our method for matrix inversion tasks and deconvolution problems, by enforcing robustness in the reconstruction through different types of Gaussian perturbations (isotropic and anisotropic). 
    \end{enumerate}
\subsection{Related works}

Distributional Robust Optimization approaches have been extensively studied in recent years \cite{blanchet2024distributionally}. While several works focused on Optimal Transport approaches \cite{mohajerin2018data,kuhn2019wasserstein,blanchet2019quantifying,blanchet2019robust}, a vast literature exists also for different metrics and divergences \cite{zhu2021kernel,hu2013kullback,blanchet2023unifying,wang2021sinkhorn}. Bayesian variants of DRO have been explored in several works \cite{levy2012robust,zhang2022distributionally}. Another approach would be to consider DRO with data augmentation \cite{sinha2020certifying}, but we aim to give a model-based perspective with more flexibility and theoretical guarantees. Finally, we mention that DRO frameworks have proven useful in data-driven contexts \cite{kuhn2019wasserstein,sagawa2019distributionally} and to defend against adversarial attacks \cite{bui2022unified,levine2020wasserstein}. Other relevant applications of DRO include Kalman filtering \cite{shafieezadeh2018wasserstein} and fairness \cite{wang2024wasserstein}.

\section{Preliminaries}

\subsection{Distributional Robust Optimization (DRO)}
Distributional Robust Optimization (DRO) frameworks 
propose an adversarial approach that, given an estimated in-sample distribution $\mu^*$, can find the best parameter estimation in $\Theta$ that is robust to out-of-sample distributions that belong to a given set of perturbations.
This is achieved by considering a min-max optimization problem, where the optimal parameter is found by solving 
\begin{align}\label{eq:DRO}
    \min_{g \in \Theta} \max_{\mu \in B_\varepsilon(\mu^*)} \int_S \ell(s,g) d\mu(s),
\end{align}
with a measurable loss $\ell : S \times \Theta \rightarrow [0,\infty)$ and where $B_\varepsilon(\mu^*)$ is the set of perturbations around the given estimated distribution $\mu^* \in P(S)$. Note that this is typically defined as the $\varepsilon$-ball given a suitable discrepancy $D: P(S) \times P(S) \rightarrow \R$ between probability measures.

\subsection{Wasserstein distances and entropy regularization}


Given a distance $c: S\times S \rightarrow [0,\infty]$, the Wasserstein-1 distance $W_1 : P(S) \times P(S) \rightarrow [0,\infty)$ between two probability measures $\mu$ and $\nu$ is defined as
\begin{align}
W_1(\mu,\nu) =   \inf_{\pi \in \Pi(\mu, \nu)}  \int_{S \times S} c(s,r) d\pi(s,r),
\end{align}
where $\Pi(\mu, \nu)$ is the set of admissible transport plans defined as the probability measures in $S\times S$ with marginals $\mu$ and $\nu$.
Moreover, 
we consider the entropy-regularized Wasserstein-1 distance \cite{peyre2019computational}
,  defined as
\begin{align}\label{eq:entropy}
    W_1^\delta(\mu, \nu) &:= \inf_{\pi \in \Pi(\mu, \nu)}  \int_{S \times S} c(s,  r) + \delta \log\left( \frac{d\pi}{d\eta}\right) d\pi(s, r)
\end{align}
where $\delta > 0$ is a regularization parameter and $\eta \in P(S\times S)$ is a reference measure. Note that the previous minimization problem is performed on the transport plans $\pi \in \Pi(\mu, \nu)$ that are absolutely continuous with respect to $\eta$. 
It is well-known that the optimization problem \eqref{eq:entropy} is strictly convex in $\pi \in \Pi(\mu,\nu)$ due to the addition of the entropy term. This important property allows for explicit optimality conditions \cite{peyre2019computational,wang2021sinkhorn}, which enable the introduction of fast algorithms \cite{cuturi2013sinkhorn}.
For this reason, 
we use $W_1^\delta(\mu, \nu)$ as the discrepancy in our framework.

\subsection{Inverse problems}

The main goal of this paper is to apply the perturbation-aware DRO problem to inverse problems.
For simplicity, assume we have a linear operator $H: X \rightarrow Y$, representing the data acquisition process, whose output is corrupted by noise as
\begin{align}\label{eq:inverse}
y = Hx + {\rm noise}.   
\end{align}
The goal of inverse problems is to reconstruct the true parameter $x\in X$ from the noisy measurements $y \in Y$ produced as in \eqref{eq:inverse}. Crucially, this reconstruction process is often ill-posed. This is due to $H$ possibly being ill-conditioned or not injective, and the sensitivity of the reconstruction process with respect to the presence of noise \cite{engl1996regularization}. To solve these issues, many approaches have been proposed, both based on classical optimization tools and data-driven techniques \cite{carioni2023unsupervised2,benning2018modern}. Notably, one can aim at approximating a \emph{good} reconstruction of $x$ by solving a variational problem of the form
\begin{align}
    \inf_{x\in X} \frac{1}{2}\|Hx - y\|_2^2 + R(x) 
\end{align}
where $R: X \rightarrow [0,\infty]$ is a suitable regularizer, responsible for encoding a priori information in the reconstruction. The most well-known example of such regularization is Tikhonov-regularization, where $R(x)=\|x\|_2^2$. Other notable approaches are based on a Bayesian interpretation of the problem \eqref{eq:inverse} see, e.g. \cite{dashti2013bayesian}, where the measurement process \eqref{eq:inverse} induces a joint distribution $P(X\times Y)$ of parameters and (noisy) measurements and the goal is to construct the full posterior distribution of $x$ given $y$.

In this work, we propose an alternative framework based on robustness against perturbations in $S= X \times Y$.

\section{A perturbation-aware framework for DRO}

This section is devoted to defining our novel framework for perturbation-aware distributional robustness.
First, we prescribe the set of admissible perturbations by fixing an a-priori family of perturbations $K \subset P(S)$. The set $K$ includes all perturbations of the in-sample estimated distribution $\mu^* \in P(S)$ that could be employed by an adversary in the worst-case optimization problem in \eqref{eq:DRO}. 
More precisely, we consider the set of admissible perturbations of the measure $\mu^*$ that are constructed as the second marginal of any measure $\nu \in P(S\times S)$ of the form 
\begin{align}
    \nu = \mu^* \otimes \pi_s, 
\end{align}
where $\pi_s : S \rightarrow P(S)$ is any conditional distribution such that $\pi_s \in K$ for $\mu^*$-a.e. $s\in S$. We denote such a space by
\begin{align*}
    B_{K} = \{\mu\in \mathcal{M}_+(S),\,  2^{nd}\text{ marginal of } \mu^* \otimes \pi_s: \pi_{s}\in K \text{ for } \mu^*\text{-a.e. } s \in S\}.
\end{align*}
We then consider only admissible perturbations that are inside an $\varepsilon$-ball centered in $\mu^*$ according to a given discrepancy $D : P(S) \times P(S) \rightarrow [0,\infty]$.
This leads us to define the following sets of admissible perturbations
\begin{align*}
    B_{D, \varepsilon, K}(\mu^*) &:= B_K \cap \{ \mu \in P(S) : D(\mu^*,\mu) \leq \varepsilon\}.
\end{align*}

In this paper, we aim to study distributionally robust optimization (DRO) settings for perturbations defined as above. Therefore, given a parameter space $\Theta$ and a measurable loss $\ell : S \times \Theta \rightarrow [0,\infty)$ we define the perturbation-aware DRO problem as follows.
\begin{defi}\label{def:constrdro}
    The perturbation-aware DRO problem is defined as 
\begin{align}\label{eq:pdro}
\inf_{g \in \Theta}\sup_{\mu \in  B_{D, \varepsilon, K}(\mu^*)} \int_S \ell(s,g) d\mu(s).
\end{align}
Moreover, we choose as $D$ the entropy-regularized Wasserstein-1 distance defined in \eqref{eq:entropy}. This leads to the perturbation-aware $W^\delta_1$-DRO problem 
\begin{align}\label{eq:entrDRO}
\inf_{g \in \Theta}\sup_{\mu \in  B_{W_1^\delta, \varepsilon, K}(\mu^*)} \int_S \ell(s,g) d\mu(s).
\end{align}
\end{defi}

We remark that the choice of $W^\delta_1$ as discrepancy enables the use of efficient algorithms for DRO optimization that rely on the specific properties of the entropic regularization \cite{wang2021sinkhorn}. 

\subsection{Perturbation-aware DRO meets inverse problems}\label{sec:pert}
When dealing with inverse problems, that is, when $S = X \times Y$, the choice of admissible perturbations $K \subset P(X\times Y)$ should reflect the type of the robustness that we want to enforce in the solution; so this is naturally a problem-dependent decision. 
In particular, we adopt a Bayesian interpretation of problem \eqref{eq:inverse} and assume that this induces a joint distribution $P(X\times Y)$ over the joint parameter/measurement space. 
Therefore, when searching for a posterior, we can look for the worst-case scenarios for different, suitably chosen sets $K$. 
In particular, we pinpoint four different classes of perturbations that are worth considering:
\begin{enumerate}
    \item \textbf{$X$-perturbations.} This case models perturbations of the input distribution (e.g. equivariant transformations, frequency modifications or shifts). The set $K$ is chosen so that $K \subset P(X)$.
    
    \item \textbf{$Y$-perturbations.} This case models perturbations of the full measurements distribution, or, equivalently, transformations of the output distribution. 
    The set $K$ is chosen so that $K \subset P(Y)$. 
    
    \item \textbf{$Y|X$-perturbations.}
    This case models perturbations in the conditional process $Y|X$ and is therefore of crucial importance for inverse problems. One main example of such perturbations is given by the additive noise in the measurements, c.f. \eqref{eq:inverse}. This is particularly relevant since it mimics the data acquisition process and can be adapted to different types of noise that might be present in the measurements. 
    Moreover, any other form of conditional modification of measurements, such as imperfect forward measurements \cite{bungert2020variational},  belongs to this class.  The set $K$ is chosen as $K \subset \{\mu \in P(X \times Y): \mu_X^* \in P(X) \text{ is the } X\text{-marginal of } \mu\}$. 
    
    \item \textbf{$X|Y$-perturbations.} This case is similar to the mentioned $Y|X$-perturbations, but now only perturbations in the inverse conditional process $X|Y$ are considered instead. The set $K$ is chosen as  $K \subset \{\mu \in P(X \times Y): \mu_Y^* \text{ is the } Y\text{-marginal of } \mu$\}.
\end{enumerate}

In this paper, we will focus on examples from  class 3, due to their relevance for inverse problems and their novelty compared to traditional DRO frameworks.
 \section{Upper bounds and weak duality results}

In this section we derive duality results for the worst-case term in Definition \ref{def:constrdro}:
\begin{align}
I := \sup_{\mu \in  B_{W_1^\delta, \varepsilon, K}(\mu^*)} \int_S l(s, g) d\mu(s),
\end{align}
where we suppress the dependence on $g$ in the definition of $I$.
Apart from its theoretical interest, such a result will be fundamental to design an implementable algorithm to solve \eqref{eq:entrDRO}.
We prove the following weak duality result, under the choice of the reference measure $\eta = \mu^* \otimes \eta_s$ for a given $\eta_s : S \rightarrow P(S)$.
\begin{prop}\label{prop:weak}
It holds that 
\begin{align}
    I \geq \inf_{\lambda \leq 0} \lambda \varepsilon + \int_S h_{s}(\lambda) d\mu^*(s)
\end{align}
where
\begin{equation}\label{eq:h}
    h_{s}(\lambda) := \sup_{\mu \in K} \int_S l (r,g) - \lambda c(s , r) - \lambda \delta \log \left(\frac{d\mu}{d\eta_s}\right)d\mu(r).
\end{equation}
\end{prop}

\begin{proof}
 Note that we can write $I$ as 
\begin{align*}
I = \sup_{\pi_s} \int_{S \times S} \ell(r,g) d\pi_s(r) d\mu^*(s)
\end{align*}
where $\pi_s \in K$ for $\mu^*-a.e$, $s\in S$ and  $\int_S \int_S c(s,r) + \delta \log\left(\frac{d\pi_s}{d\eta_s}\right) d\pi_s(r) d\mu^*(r) \leq \varepsilon$. Introducing the Lagrange multiplier $\lambda \geq 0$ for such constraint  we obtain  that 
\begin{align*}
I & \leq \sup_{\pi_s \in K} \inf_{\lambda \geq 0}\left[ \lambda \varepsilon + \int_{S \times S} \ell(s,g) - \lambda c(s,r) - \lambda \delta \log\left(\frac{d\pi_s}{d\eta_s}\right) d\pi_s(r) d\mu^*(s)  \right]\\
& \leq \inf_{\lambda \geq 0}\sup_{\pi_s \in K}\left[ \lambda \varepsilon + \int_{S \times S} \ell(s,g) - \lambda c(s,r) - \lambda \delta \log\left(\frac{d\pi_s}{d\eta_s}\right) d\pi_s(r) d\mu^*(s)  \right]\\
& = \inf_{\lambda \geq 0} \left[\lambda \varepsilon +  \sup_{\pi_s \in K}\int_{S} \left(\int_S \ell(r,g) - \lambda c(s,r) - \lambda \delta \log\left(\frac{d\pi_s}{d\eta_s}\right) d\pi_s(r) \right)d\mu^*(s)  \right]\\
& \leq \inf_{\lambda \geq 0} \left[\lambda \varepsilon +  \int_{S} \left( \sup_{\mu \in K} \int_S \ell(r,g) -\lambda  c(s,r) - \lambda \delta \log\left(\frac{d\mu}{d\eta_s}\right) d\mu(r) \right)d\mu^*(s)  \right]
\end{align*}
as we wanted to prove.
\end{proof}

\begin{rem}
    In \cite{wang2021sinkhorn}, it has been shown that if $K = P(S \times S)$, meaning no perturbation constraints are imposed, then strong duality holds and $h_s(\lambda)$ has an explicit expression given by 
    \begin{align}
        h_s(\lambda) = \lambda \delta \int_S e^{\frac{\ell(r, g) - \lambda c(s,r)}{\lambda \delta}} d\eta_s(r).
    \end{align}
\end{rem}
\section{Numerics}
\subsection{Algorithm description}
To compute the optimal $g \in \Theta$ solving \eqref{eq:entrDRO} we rely on the (weak) dual formula for the perturbation-aware DRO problem derived in Proposition \ref{prop:weak}. In particular, we numerically solve the following optimization problem
\begin{align}\label{eq:upperbound}
    \inf_{g \in \Theta}\inf_{\lambda \leq 0} \lambda \varepsilon + \int_S h_{s}(\lambda) d\mu^*(s)
\end{align}
where $h_{s}$ is defined as in \eqref{eq:h}. Thanks to Proposition \ref{prop:weak}, this is an upper bound for \eqref{eq:entrDRO}. Note that, even if we will not prove it in the current paper, we believe that under suitable assumptions on the set $K$ and the loss $\ell$ a strong duality result would hold. 

The algorithm used to compute solutions of \eqref{eq:upperbound} is adapted from \cite{wang2021sinkhorn}.
Crucially, we assume that the set $K$ can be parametrized by a vector $q \in Q$ and the objective is differentiable with respect to $q$. This will be true for all the numerical examples considered in this paper.
We perform a bisection search algorithm to determine the optimal $\lambda$, while the optimal $g$ and $q$ for a given $\lambda$ are found using an alternating Biased Stochastic Mirror Descent (BSMD) algorithm. 

The bisection search algorithm iteratively narrows its search interval by evaluating the objective function at the mid- and endpoints of the current interval (using the BSMD algorithm), adjusting the bounds based on which point results in a smaller function value. This process continues until the interval is sufficiently small, resulting in an optimal $\lambda$.

For each $\lambda$, the BSMD algorithm iteratively finds the optimal $g$ and $q$. Every iteration, $q$ is first updated by adding its gradient scaled by its learning rate in a maximization step. Afterwards, $g$ is updated by subtracting its gradient scaled by its learning rate in a minimization step. Both gradient updates are estimated by a Randomized Truncation MLMC (RT-MLMC) estimator, which provides a biased gradient estimate with controlled bias. For more details on the algorithm, we refer to \cite{wang2021sinkhorn}.

\subsection{Isotropic Gaussian perturbations for matrix inversion}\label{sec:gaussian_iso}

We consider as first example the case where $K$ is defined as
\begin{align*}
K = \{\mu_X^* \otimes \mathcal{N}(Hx,\sigma) : 0<  \sigma < M\}
\end{align*}
where $N(p,\sigma)$ is a Gaussian with average $p$ and covariance matrix $\sigma \cdot {\rm Id}$, $M$ is a fixed positive constant and $\mu^*_X$ is a given empirical distribution in $X$
\begin{align}
    \mu_X^* = \frac{1}{N}\sum_{i=1}^N \delta_{x_i}.
\end{align}
Note that this choice of perturbation belongs to class $3$ described in Section \ref{sec:pert} and enforces robustness against Gaussian noise.

In this experiment, we simply aim to perform matrix inversion. To this end, we consider $H$ to be a given matrix in $\R^{n \times m}$, 
$\Theta = \R^{n\times m}$ and $\ell((x,y),g) = \|g(y) - x\|^2_2$ for $g \in \R^{n\times m}$ and we aim at solving \eqref{eq:entrDRO} to find the inverse of $H$ from an initial in-sample estimation $\mu^* \in P(X\times Y)$ given by 
\begin{align}\label{eq:emp}
    \mu^* = \frac{1}{N}\sum_{i=1}^N \delta_{(x_i,y_i)}
\end{align}
where $y_i = Hx_i + {\rm noise}$.
Note that, by choosing as reference measure $\nu = \mu_X^*\otimes dy$ one can compute $h_{s} = h_{(\bar x, \bar y)}$ explicitly as
\begin{align}\label{eq:h}
  h_{(\bar x, \bar y)}(\lambda)  =   \sup_{\sigma \in (0,M]} \int_{X \times Y} & \|x - g(y)\|_2^2 - \lambda c((\bar x,\bar y), (x,y)) \\
  & \qquad - \lambda \delta \log \left[N(Hx, \sigma)\right]\, N(Hx, \sigma)(y) dy d\mu_X^*(x). \nonumber
\end{align}

We test the algorithm for the simple case of inverting the $2\times 2$ matrix given as $H = 2\cdot {\rm Id}$. We consider $(x_i)_{i=1}^{400}$ sampled uniformly in the square $[0,1]^2$ and we generate $y_i$ as $Hx_i$. The cost used inside $W_1^\delta$ is simply the Euclidean distance, while the parameters are set as $\delta = 0.1$ and $\varepsilon = 0.001$. The optimal $g_{opt}$ and $\sigma_{opt}$ obtained by the algorithm are
\begin{align}
    g_{opt} = \left[\begin{array}{cc}
     0.442 \ & \ 0.039 \\
 0.038 \ & \ 0.450  
    \end{array}
    \right]
\quad \sigma_{opt} = 0.280
\end{align}
and the convergence graph is reported in Figure \ref{fig:isotropic}.

\begin{figure*}[h!]
	\centering
\begin{subfigure}
{.49\textwidth}
\includegraphics[height=1.9in]{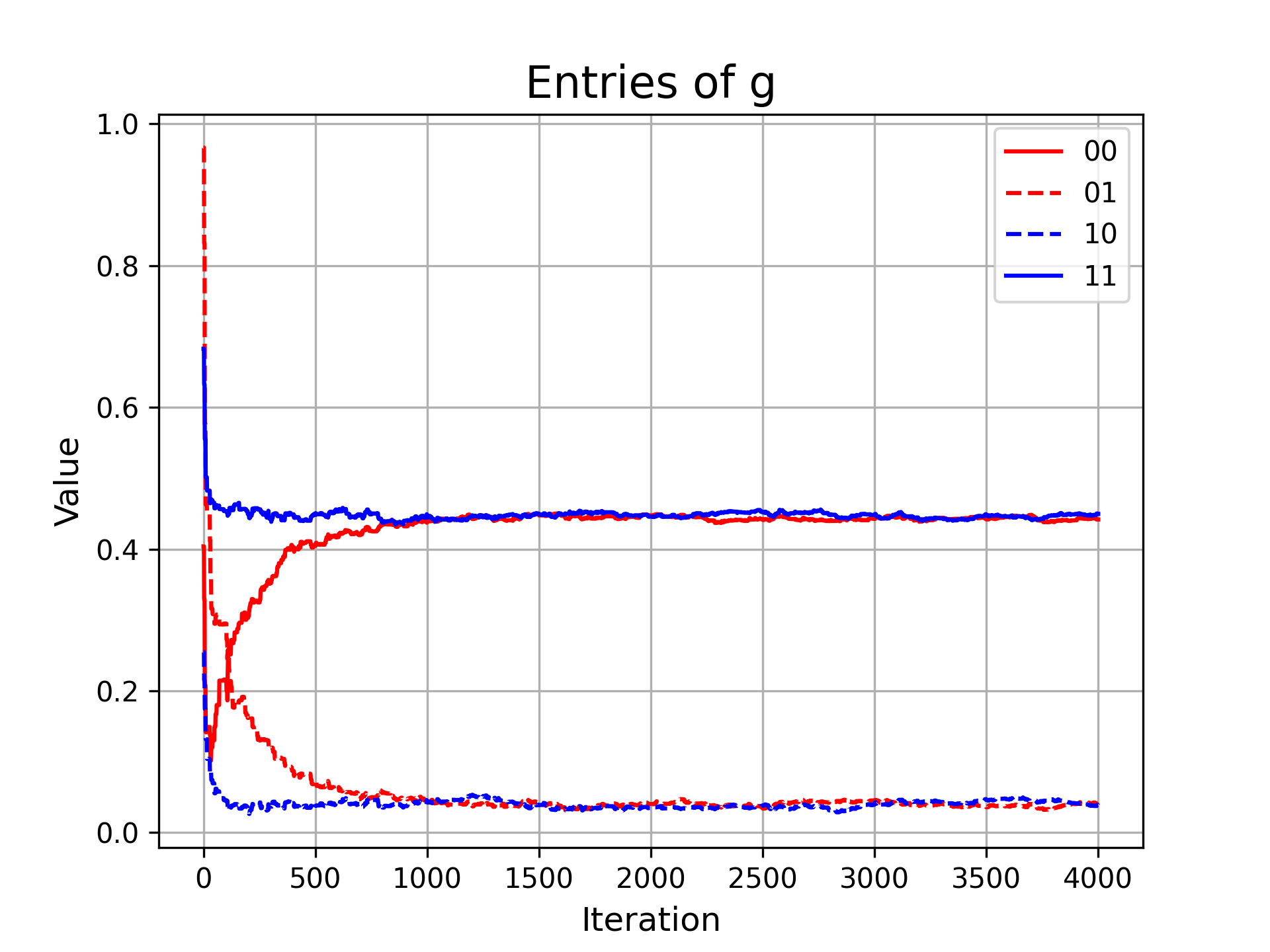}
\end{subfigure}
\begin{subfigure}{.49\textwidth}
\includegraphics[height=1.9in]{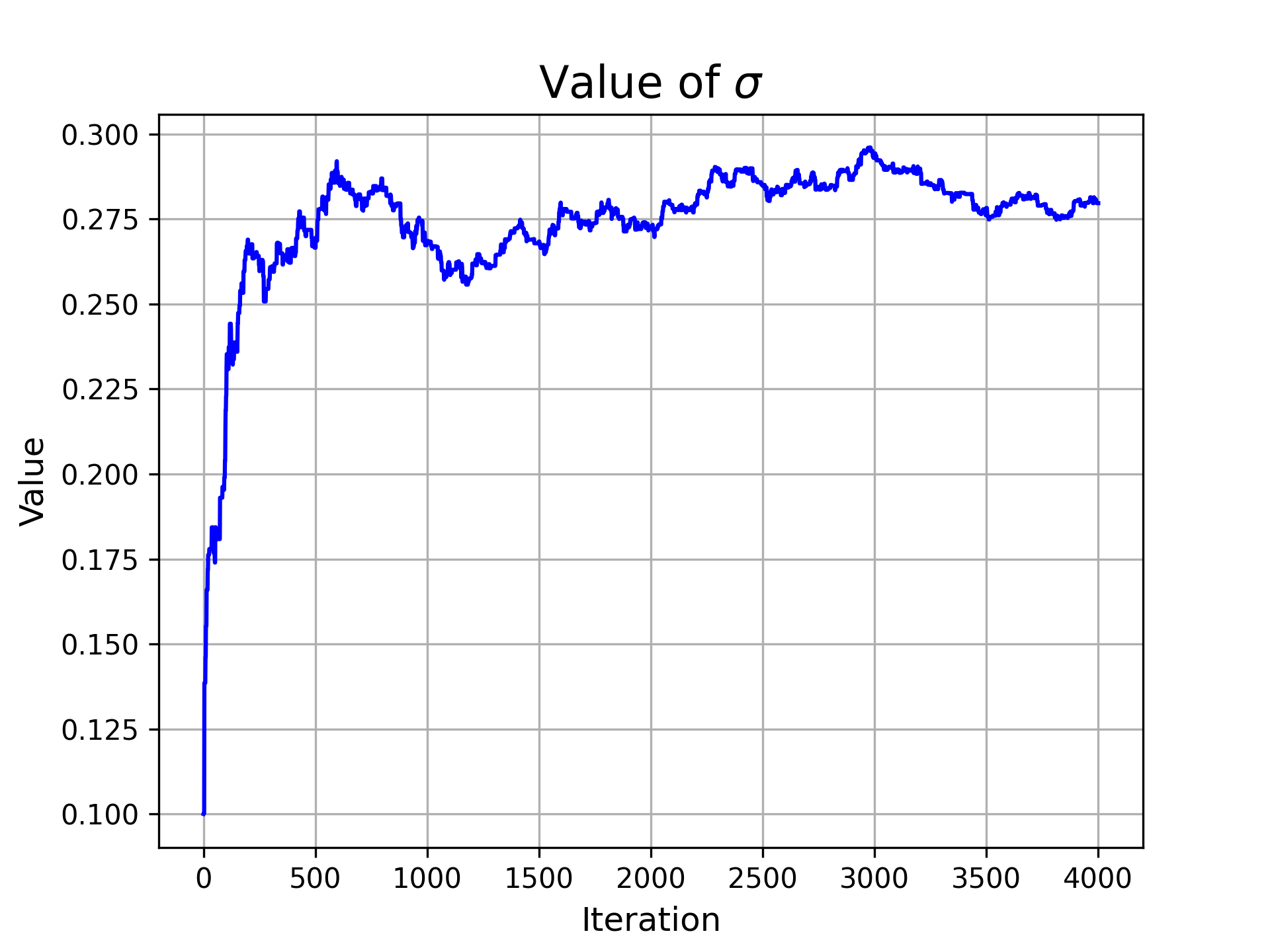}
	\end{subfigure}
	\caption{Left: entries of the reconstructed $g$ through the iterations. Right: values of the standard deviation $\sigma$ through the iterations.} \label{fig:isotropic}
 \end{figure*}
Note that the discrepancy between the reconstructed inverse and the true inverse has to be expected as a result of the regularization effect of the robust framework. This has been shown in \cite[Proposition 2]{blanchet2019robust} for linear regression and we expect a similar effect in our model, resulting in an attenuation of the spectral norm of $g_{opt}$.

\subsection{Anisotropic Gaussian perturbations for matrix inversion}
In this example we consider the case where $K$ is defined as
\begin{align*}
K = \{\mu^* \otimes \mathcal{N}(Hx,\Sigma) : 0<  \Sigma < M\}
\end{align*}
where $N(p,\Sigma)$ is a Gaussian with average $p$ and covariance matrix $\Sigma$, $M$ is a fixed positive constant and $\mu^*_X$ is a given empirical distribution in $X$. Note that also this case belongs to class $3$ described in Section \ref{sec:pert}. However, it enforces robustness with respect to noise that exhibits anisotropic behaviour. 
Similarly to Section \ref{sec:gaussian_iso}, given a matrix $H$ in $\R^{n\times m}$ we perform matrix inversion by solving \eqref{eq:entrDRO} for $\ell((x,y),g) = \|g(y) - x\|^2_2$ and $\mu^* \in P(X\times Y)$ given as in \eqref{eq:emp}.

Similarly to the isotropic case, by choosing as reference measure $\nu = \mu_X^*\otimes dy$ one can compute $h_{s} = h_{(\bar x, \bar y)}$ explicitly as in \eqref{eq:h}.
Note that in order to preserve the constraint $0<\Sigma < M$ in the supremum above, we factorize $\Sigma$ using its Cholesky decomposition.

We test the algorithm for a $2\times 2$ matrix given as 
\begin{align}
    H = \left[\begin{array}{cc}
     5 \ & \ 1 \\
1 \ & \ 2  
    \end{array}
    \right].
\end{align}
We consider $(x_i)_{i=1}^{600}$ sampled uniformly in the square $[0,1]^2$ and we generate $y_i$ as $Hx_i$. The cost in $W_1^\delta$ is simply the Euclidean distance, while the parameters are set as $\delta = 0.1$ and $\varepsilon = 0.001$. The optimal $g_{opt}$ and $\Sigma_{opt}$ are
\begin{align}
    g_{opt} = \left[\begin{array}{cc}
     0.195 \ & \ -0.0607 \\
 -0.0309 \ & \ 0.380
    \end{array}
    \right],
\quad \Sigma_{opt} = \left[\begin{array}{cc}
     0.111 \ & \ -0.000107 \\
 -0.000107 \ & \ 0.167
    \end{array}
    \right]
\end{align}
and the convergence graph is reported in Figure \ref{eq:differentalpha}.
It is worth to notice that the optimal perturbation in $K$, that is a multivariate Gaussian with covariance matrix $\Sigma_{opt}$ has anisotropy in the direction of the eigenvector of the maximal eigenvalues of $H$. This is reported in Figure \ref{covariance} and it confirms that the admissible perturbation is capturing the deformation induced by the operator $H$.

\begin{figure*}[h!]
	\centering
\begin{subfigure}
{.49\textwidth}
\includegraphics[height=1.9in]{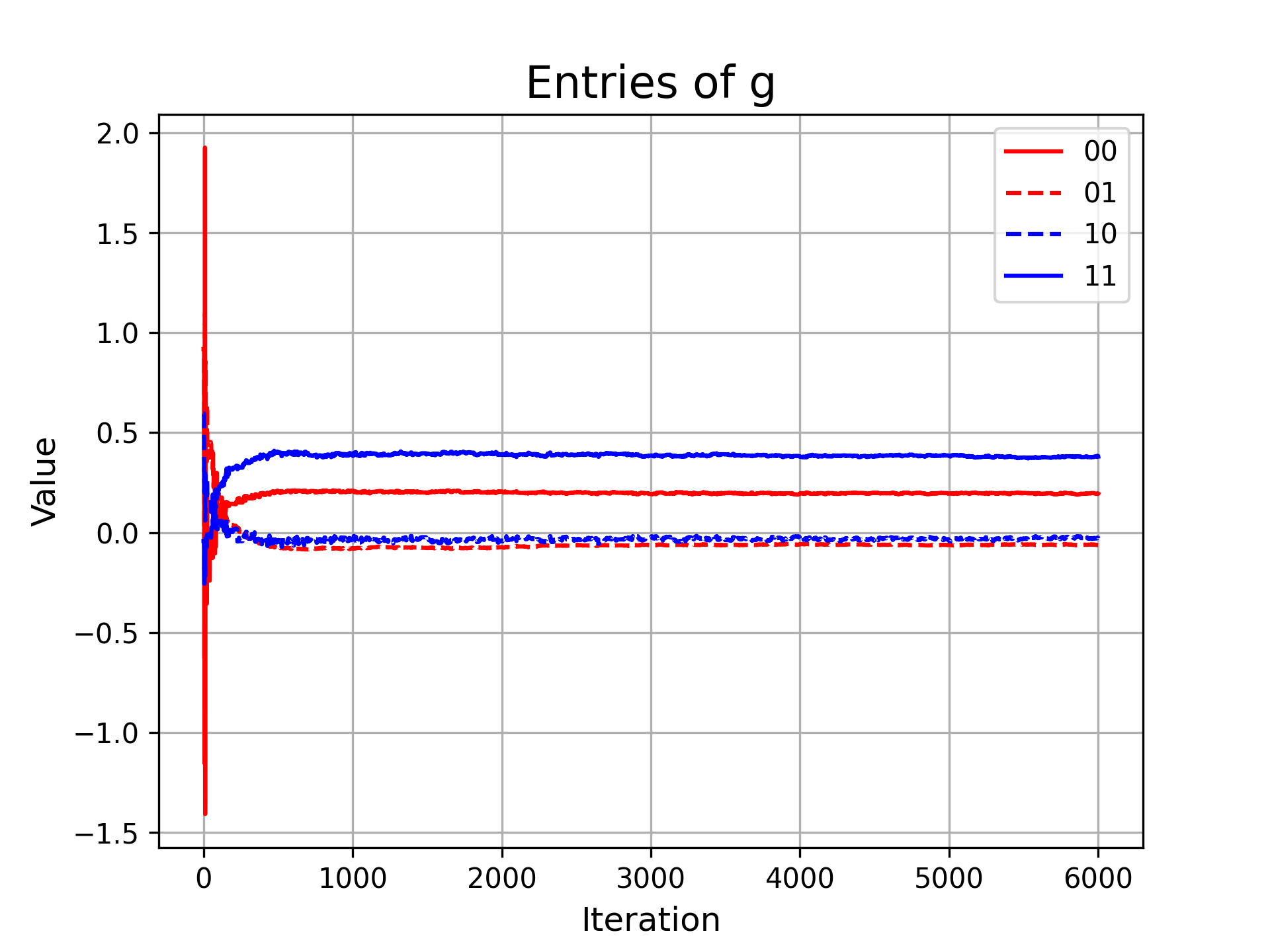}
\end{subfigure}
\begin{subfigure}{.49\textwidth}
\includegraphics[height=1.9in]{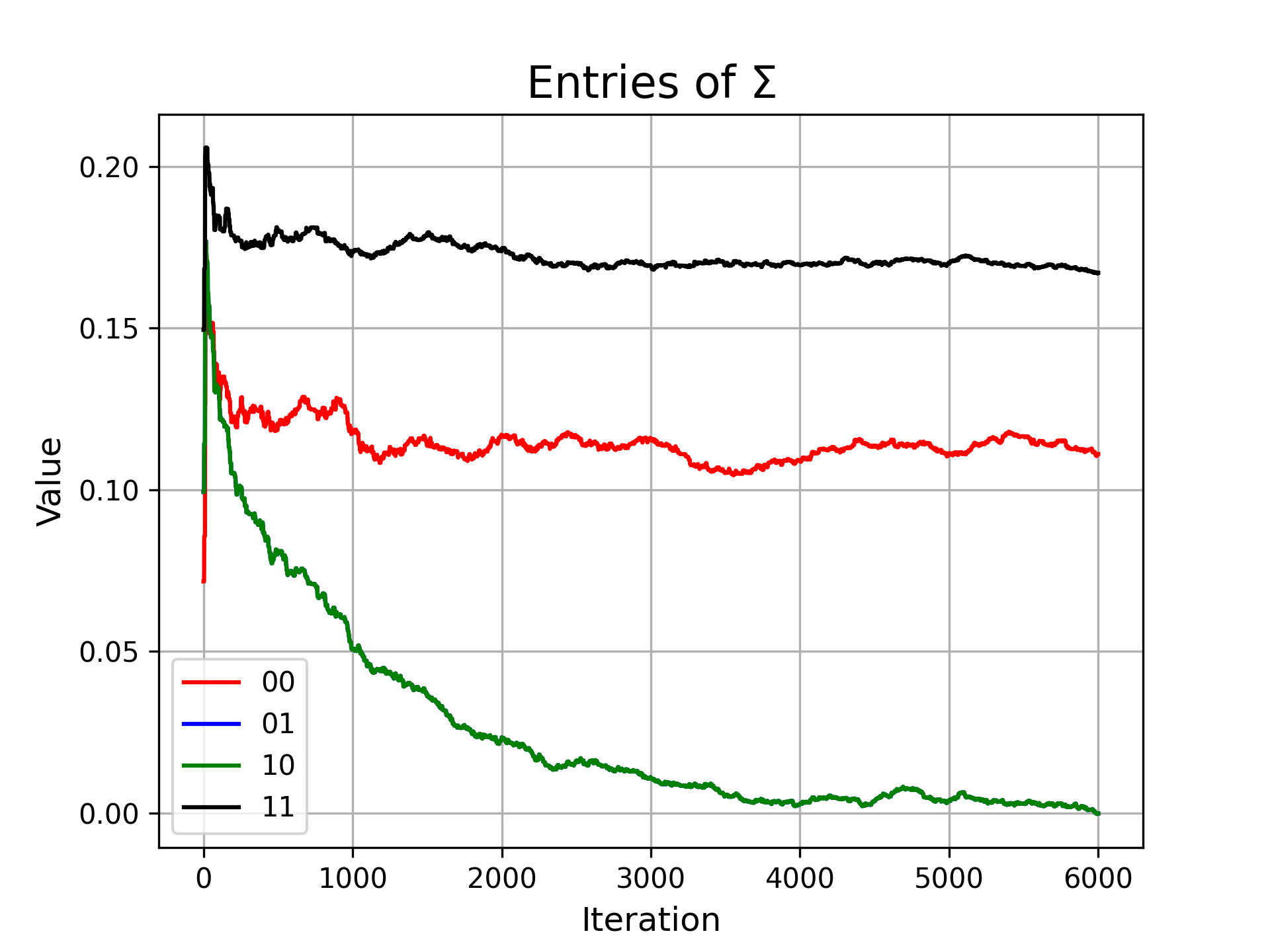}
	\end{subfigure}
	\caption{Left: entries of the matrix $g$ through the iterations. Right: entries of the covariance matrix $\Sigma$ through the iterations. Note that the lines for "01" and "10" in $\Sigma$ overlap.} \label{eq:differentalpha}
 \end{figure*}

\begin{figure*}[h!]
	\centering
\includegraphics[height=1.9in]{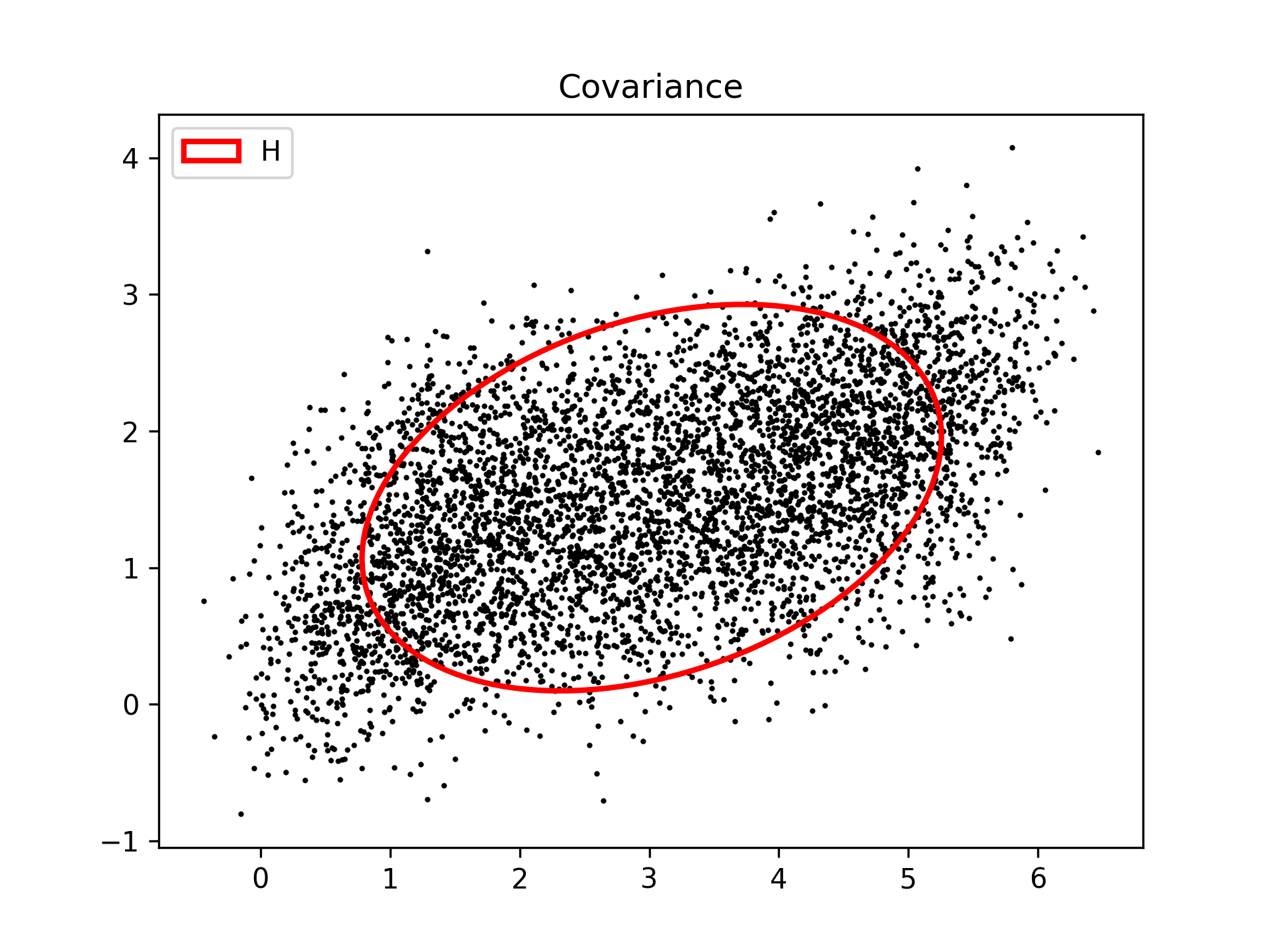}
\caption{Samples from a Gaussian with average $Hx$ and optimal covariance $\Sigma_{opt}$ together with the ellipse representing the of eigenvectors of $H$.}\label{covariance}
 \end{figure*}

\subsection{Image deconvolution}
Finally we apply the framework outlined in Section \ref{sec:gaussian_iso} (isotropic Gaussian perturbation) for image deconvolution. In particular, here $H : \R^{28 \times 28} \rightarrow \R^{26 \times 26}$  is chosen to be a convolution operator with a $3 \times 3$ Gaussian filter. Given samples $(x_i,y_i)_{i=1}^{150}$ where $y_i$ are constructed as $Hx_i$,  we construct a robust version of a deconvolution operator. We apply the optimal deconvolution operator to reconstruct 1) blurred MNIST images that are additionally corrupted by Gaussian noise of standard deviation $\sigma \in \{0.01, 0.05, 0.1\}$ and 2) blurred MNIST images corrupted by Poisson noise where the image $y_i$ is generated as $y_i \sim \text{Pois}(Hx/\sigma)$ with $\sigma \in \{0.01, 0.05, 0.1\}$ and $Hx$ representing the blurred MNIST image. Parameters are set to $\delta = 0.1$ and $\epsilon = 0.001$. The result is reported in Figure \ref{fig:dec}, compared to a reconstruction using an inverse with generalized Tikhonov regularization with the Laplacian operator and regularization parameter $\lambda$. The optimal regularization parameter $\lambda$ is determined via bisection search, minimizing the MSE computed over the entire MNIST dataset (rather than individual images). Error metrics based on $10000$ MNIST images are reported in Table \ref{tab:metrics}. 

 \begin{table}[h!]
\centering
\caption{Comparison of PADRO and generalized Tikhonov deconvolutions with different noise levels ($\sigma$) based on MSE and SSIM metrics over 10000 MNIST images.}
\label{tab:metrics}

\subcaption{Gaussian noise} 

\small 
\begin{minipage}{0.31\textwidth}
\centering
\begin{tabular}{|c|c|c|}
\hline
$\sigma = 0.01$ & \textbf{MSE} & \textbf{SSIM} \\ \hline
\textbf{PADRO}  &    0.0132           &    0.821           \\ \hline
\textbf{Tikhonov} &    0.0137          &      0.754         \\ \hline
\end{tabular}
\end{minipage}
 \ \ \ 
\begin{minipage}{0.31\textwidth}
\centering
\begin{tabular}{|c|c|c|}
\hline
$\sigma = 0.05$ & \textbf{MSE} & \textbf{SSIM} \\ \hline
\textbf{PADRO}  &       0.0195       &  0.707             \\ \hline
\textbf{Tikhonov} &       0.0230       &       0.682        \\ \hline
\end{tabular}
\end{minipage}
 \ \ \ 
\begin{minipage}{0.32\textwidth}
\centering
\begin{tabular}{|c|c|c|}
\hline
$\sigma = 0.1$  & \textbf{MSE} & \textbf{SSIM} \\ \hline
\textbf{PADRO}  &      0.0307        &   0.621            \\ \hline
\textbf{Tikhonov} &     0.0340         &      0.608         \\ \hline
\end{tabular}
\end{minipage}

\medskip

\subcaption{Poisson noise} 

\begin{minipage}{0.31\textwidth}
\centering
\begin{tabular}{|c|c|c|}
\hline
$\sigma = 0.01$ & \textbf{MSE} & \textbf{SSIM} \\ \hline
\textbf{PADRO}  &     0.0160         &      0.846         \\ \hline
\textbf{Tikhonov} &    0.0176          &      0.725         \\ \hline
\end{tabular}
\end{minipage}
\ \ \ 
\begin{minipage}{0.31\textwidth}
\centering
\begin{tabular}{|c|c|c|}
\hline
$\sigma = 0.05$ & \textbf{MSE} & \textbf{SSIM} \\ \hline
\textbf{PADRO}  &      0.0222        &   0.779            \\ \hline
\textbf{Tikhonov} &      0.0241        &      0.668         \\ \hline
\end{tabular}
\end{minipage}
\ \ \ 
\begin{minipage}{0.32\textwidth}
\centering
\begin{tabular}{|c|c|c|}
\hline
$\sigma = 0.1$  & \textbf{MSE} & \textbf{SSIM} \\ \hline
\textbf{PADRO}  &      0.0254        &   0.742            \\ \hline
\textbf{Tikhonov} &      0.0289        &     0.619           \\ \hline
\end{tabular}
\end{minipage}

\end{table}

\normalsize

 \begin{figure*}[h!]
	\centering
\begin{subfigure}{0.49\textwidth}
\centering
\includegraphics[width=\textwidth]{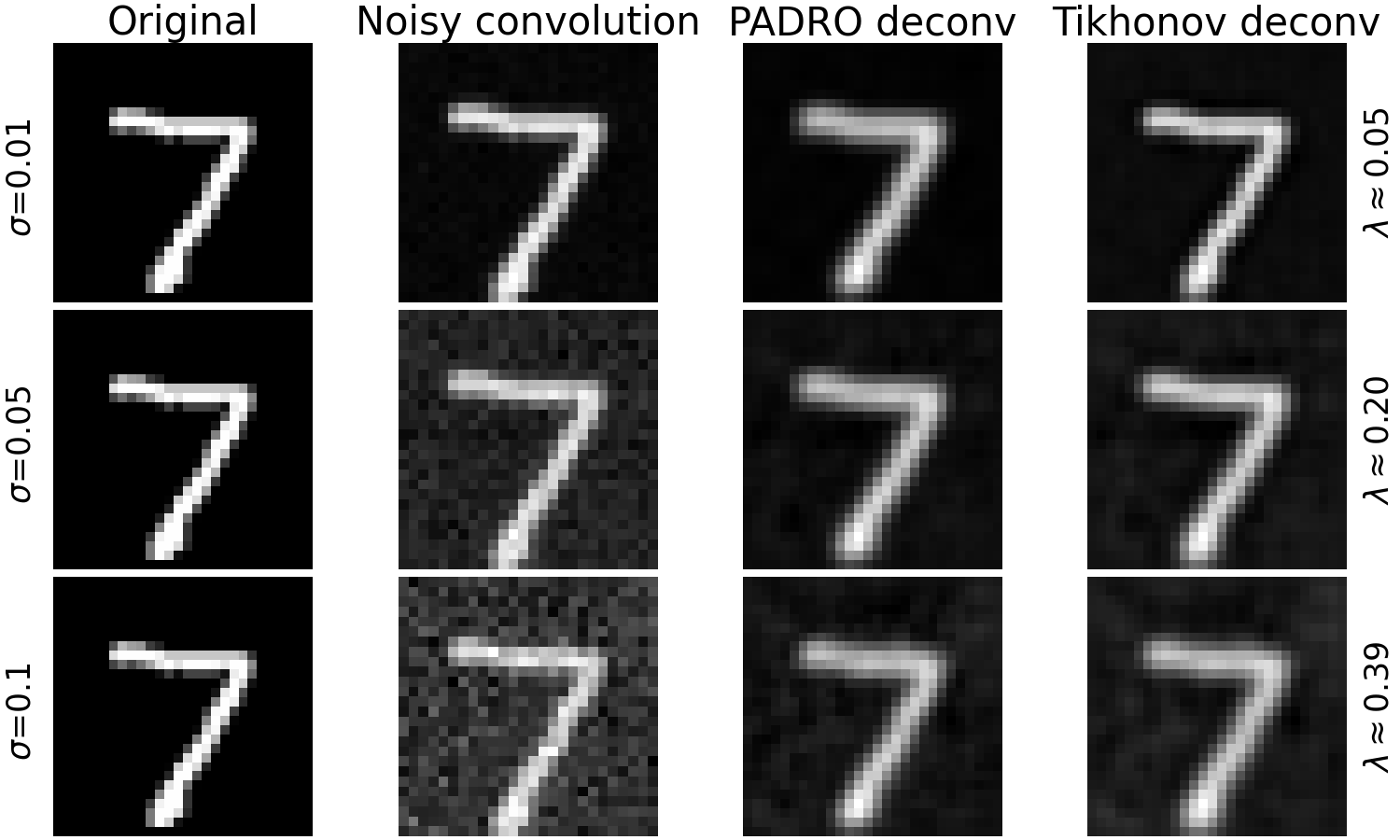}
\caption{Gaussian noise}
\end{subfigure}
\begin{subfigure}{0.49\textwidth}
\centering
\includegraphics[width=\textwidth]{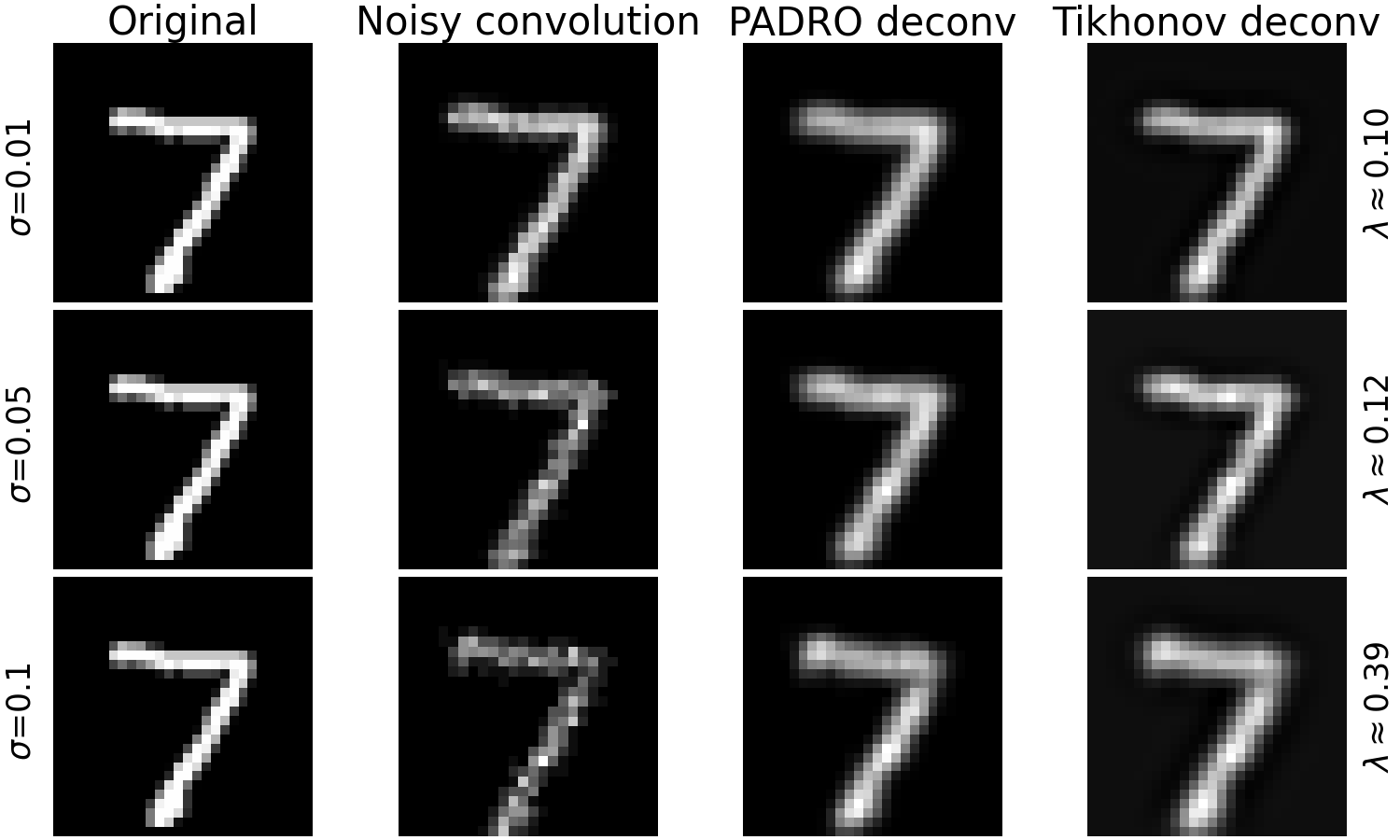}
\caption{Poisson noise}
	\end{subfigure}
	\caption{Deconvolution result for a given MNIST image with different noise types and levels, compared to a Tikhonov regularized inverse.} 
    \label{fig:dec}
 \end{figure*}

Our framework results in a blurred reconstruction of the original MNIST image, this is an effect of the low-complexity inverse we consider. This disadvantage is compensated for by the fact that our framework can handle different types and levels of noise with the same optimal deconvolution operator. It is also seen that our framework has very similar results to the Tikhonov regularized inverse. However, the advantage of our framework is that we learn the inverse, which can now be applied to any MNIST image corrupted by any type of noise, while model-based regularization learns the best reconstruction for each input individually. Our framework slightly outperforms the Tikhonov regularized inverse in all error metrics reported in Table \ref{tab:metrics}.

\section{Conclusion}
We have introduced the novel \textit{perturbation-aware} DRO framework which can be used to solve inverse problems and prescribe different types of robustness. This results in a flexible inverse that can be applied to measurements corrupted by any type of noise. Moreover, we proved a weak duality result, allowing us to demonstrate the method numerically for matrix inversion and deconvolution problems with promising results.


%
%
%
\bibliographystyle{splncs04.bst}
\bibliography{mybibliography}
\end{document}